\documentclass{amsart}
\usepackage{amssymb,latexsym}
\theoremstyle{plain}
\newtheorem{theorem}{Theorem}
\newtheorem{corollary}{Corollary}
\newtheorem{proposition}{Proposition}
\newtheorem{lemma}{Lemma}
\theoremstyle{definition}

\newtheorem{remark}{Remark}
\newtheorem{question}{Question}
\date{}

\begin{document}

\title[blowing up]
{A splitting criterion for vector bundles on blowing ups of the plane}
\author{E. Ballico}
\address{Dept. of Mathematics\\
  University of Trento\\
38050 Povo (TN), Italy}
\email{ballico@science.unitn.it}
\thanks{The author was partially supported by MIUR and GNSAGA of 
INdAM (Italy).}
\author{F. Malaspina}
\address{Dip. Matematica, Universit\`{a} di Torino\\
via Carlo Alberto 10, 10123 Torino, Italy}
\email{francesco.malaspina@unito.it}

\begin{abstract}
Let $f_s: X_s \to {\bf {P}}^2$ be the blowing-up of $s$ distinct 
points and $E$ a vector bundle on $X_s$. Here we give
a cohomological criterio which is equivalent to $E \cong f_s^\ast 
(A)$ with $A$ a direct sum of line bundles.
We also some cohomological characterizations of very particular rank 
$2$ vector bundles
on ${\bf {P}}^2$.
\end{abstract}

\maketitle

\section{Introduction}\label{S1}
Fix an integer $s \ge 1$ and $s$ distinct points $P_1,\dots ,P_s\in 
{\bf{P}}^2$. Let $f_s:
X_s \to {\bf {P}}^2$ denote the blowing up of the points $P_1,\dots 
,P_s$. $\mbox{Pic}(X_s) \cong \mathbb {Z}^{s+1}$ and we will take the 
line
bundles $f_s^\ast (\mathcal {O}_{{\bf {P}}^2}(1))$ and $D_i:= 
f_s^{-1}(P_i)$, $1 \le i \le s$, as free generators
of $\mbox{Pic}(X_s)$. Set $\mathcal {O}_{X_s}(t;a_1,\dots 
,a_s):=f_s^\ast (\mathcal {O}_{{\bf {P}}^2}(t))
(\sum _{i=1}^{s} a_iD_i)$. Hence $\mathcal {O}_{X_s}(t;a_1,\dots 
,a_s) \cdot \mathcal {O}_{X_s}(z;b_1,\dots ,b_s)
= tz - \sum _{i=1}^{s} a_ib_i$. For any coherent sheaf $A$ on $X_s$
set $A(t;a_1,\dots ,a_s):= A\otimes \mathcal {O}_{X_s}(t;a_1,\dots 
,a_s)$. Set $X_0:= {\bf {P}}^2$ and $f_0:= \mbox{Id}_{{\bf {P}}^2}$.
In section \ref{S2} we will prove the following result.

\begin{theorem}\label{a3}
Let $E$ be a rank $r$ torsion free sheaf on $X_s$ which is locally 
free at each point
of $D_1\cup \cdots \cup D_s$. For every $i\in \{1,\dots ,s\}$ let 
$b_{i, 1} \ge \cdots \ge b_{i,r}$
denote the splitting type of $E\vert D_i$. The following conditions 
are equivalent:
\begin{itemize}
\item[(a)] $H^1(X_s,E(t;0,\dots ,0)) =0$ for all $t\in \mathbb {Z}$.
\item[(b)] $b_{i,r} \ge 0$ for all $i \in \{1,\dots ,s\}$ and $f_{s 
\ast }(E)$ is isomorphic to a direct sum of line bundles.
\item[(c)] There is a direct sum $A$ of $r$ line bundles on ${\bf 
{P}}^2$ such that $E \cong f_s^\ast (A)$.
\end{itemize}
\end{theorem}

\begin{remark}\label{a3.0}
We will also check that every torsion free sheaf $E$ on $X_s$ such
that $H^1(X_s,E(t;0,\dots ,0)) =0$ for all $t\ll 0$ is locally free 
(see Remark \ref{c0}). Hence we may apply Theorem \ref{a3}
to this sheaf, without imposing that $E$ is locally free at each 
point of $D_1\cup \cdots \cup D_s$.
\end{remark}

It is essential that we use all multiples (positive and negative) of 
a `` minimal '' line bundle
$\mathcal {O}_{X_s}(1;0,\dots ,0)$. In section \ref{S3} we will see 
in the rank $2$ case and for the plane $X_0$
what happens if we take e.g. only twists by even degree line bundles 
(see Propositions \ref{c1} and \ref{c2.1}
and Remark \ref{c2.2}).

We work over an algebraically closed field $\mathbb {K}$.

\section{Proof of Theorem \ref{a3}}\label{S2}

\begin{remark}\label{c0}
Let $X$ be a smooth and connected projective surface and $R\in 
\mbox{Pic}(X)$ such that $\vert R\vert$
contains the sum of an effective divisor and of an ample divisor. Let 
$E$ be a torsion free
sheaf on $X$ such that $h^1(X,E\otimes R^{\otimes t})=0$
for infinitely many negative integers $t$. Consider the exact sequence
\begin{equation*}
0 \to E \to E^{\ast \ast} \to E^{\ast \ast }/E \to 0
\end{equation*}
Since $h^0(X,E^{\ast \ast }\otimes R^{\otimes t})=0$ for $t \ll 0$, 
$E$ is locally free.
\end{remark}

\begin{remark}\label{a1}
Notice that $h^i(X_s,\mathcal {O}_{X_s}(t;0,\dots ,0)) = h^i({\bf 
{P}}^2,\mathcal {O}_{{\bf {P}}^2}(t))$
for all $i=0,1,2$, and all $t\in \mathbb {Z}$. Hence 
$h^1(X,s,\mathcal {O}_{X_s}(t;0,\dots ,0)) =0$ for all $t\in \mathbb 
{Z}$.
When $t \ge 0$ we have $h^0(X_s,\mathcal {O}_{X_s}(t;a_1,\dots ,a_s)) 
= h^0({\bf {P}}^2, \mathcal {O}_{{\bf {P}}^2}(t))$ if and only if
$a_i \ge 0$ for all $i$.
\end{remark}

\begin{remark}\label{a2}
Let $U$ be a smooth and connected quasi-projective surface. Fix $P\in 
U$. Let $\pi : V \to U$ denote the blowing up of $P$.
Set $Y:= \pi ^{-1}(P)$. Let $\mathcal {I}$ denote the ideal sheaf of 
$Y$ in $V$. Hence $Y \cong {\bf {P}}^1$
and $\mathcal {I}/\mathcal {I}^2$ is (as an $\mathcal {O}_Y$-sheaf) 
isomorphic to the degree $1$ line bundle
of $Y$. For every integer $n\ge 0$ let $Y^{(n)}$ denote the $n$-th 
infinitesimal neighborhood of $Y$ in $V$,
i.e. the closed subscheme of $V$ with $\mathcal {I}^n$ as its indeal 
sheaf. Hence $Y^{(0)} = 0$. Let $\widehat Y$
denote the formal completion of $Y$ in $V$, i.e. the formal scheme 
$\projlim _n Y^{(n)}$. Fix
any rank $r$ vector bundle $E$ on $\widehat Y$. Let $b_1 \ge \cdots 
\ge b_r$ be the splitting type of $E\vert Y$.
For every
integer $n \ge 0$ we have an exact sequence
\begin{equation}\label{eqb1}
0 \to (\mathcal {I}/\mathcal {I}^2)^{\otimes n}\otimes E\vert Y \to 
E\vert Y^{(n)} \to E\vert Y^{(n-1)} \to 0
\end{equation}
Since $\dim (\widehat Y)=1$, we get that for every integer $n \ge 1$ 
the restriction map $H^1(Y^{(n)},E\vert
Y^{(n)}) \to H^1(Y^{(n)},E\vert
Y^{(n-1)})$ is surjective. Hence the restriction map $H^1(\widehat 
Y,E) \to H^1(Y,E\vert Y)$ is surjective. Since
$h^1(Y,E\vert Y) = 0$ if and only if $b_r \ge -1$, we get 
$H^1(\widehat Y,E) \ne 0$ if $b_r \le -2$.
Now assume $b_r \ge -1$. Since $(\mathcal {I}/\mathcal 
{I}^2)^{\otimes n}$ is a degree $n$ line bundle
on $Y \cong {\bf {P}}^1$, the tower of exact sequences (\ref{eqb1}) 
gives $h^1(Y^{(n)},E\vert Y^{(n)}) = 0$
Hence $H^1(\widehat Y,E)=0$ if and only if $b_r \ge -1$. Now assume 
that $E$ is the restriction to $\widehat {Y}$
of a vector bundle $F$ on $V$. The formal function theorem 
(\cite{h1}, III.11.1) gives
that $R^1\pi _\ast (F)=0$ if and only if $b_r \ge -1$. Since every 
fiber of the proper map $\pi$ has
dimension at most $1$, $R^j\pi _\ast (F)=0$ for all $j \ge 2$ (e.g. 
by the formulal function theorem (\cite{h1}, III.11.1,
and Nakayama's lemma). For any splitting type $b_1 \ge \cdots \ge b_r$
the sheaf $\pi _\ast (F)$ is torsion free and its restriction to 
$U\backslash \{P\}$ is locally free.
The natural map $\tau _F: \pi ^\ast \pi _\ast (F) \to F$ is an 
isomorphism outside $Y$.
It easy to check that if $b_r < 0$, then
$\pi _\ast (F)$ is not locally free. We have $b_1 = \cdots = b_r =0$ 
if and only if the
natural map $\tau _F: \pi ^\ast \pi _\ast (F) \to F$ is an isomorphism.
\end{remark}

\begin{lemma}\label{b2}
Take the set-up of Remark \ref{a2}. Let $F$ be a rank $r$ vector 
bundle on $V$ and let $b_1 \ge \cdots \ge b_r$ be the splitting
type of $F\vert Y$. Assume $b _r \ge 0$. $\pi _\ast (F)$ is locally 
free if and only if $b_1 = \cdots b_r = 0$.
\end{lemma}

\begin{proof}
We claimed the `` if '' part at the end of Remark \ref{a2}. Assume 
$b_1>0$. Since $b_r \ge 0$, $h^1(Y,(F\vert Y)\otimes
(\mathcal {I}/\mathcal {I}^2)^{\otimes n}) = 0$ for all $n \ge 0$. 
Hence (\ref{eqb1}) implies that
the restriction map $ H^0(Y^{(n)},F\vert Y^{(n)}) \to 
H^0(Y^{(n-1)},F\vert Y^{(n-1)})$ is surjective
for every $n>0$. The formal function theorem (\cite{h1}, III.11.1) 
implies that the fiber
of $\pi _\ast (F)$ over $P$ has dimension at least $h^0(Y,F\vert Y) = 
r+b_1+\cdots +b_r>r$. Since $\pi _\ast (F)$
has rank $r$, it is not locally free.
\end{proof}

\begin{remark}\label{b3}
Let $F$ be any vector bundle on $U$. Since $\dim (U)=\dim(V)=2$, 
$H^i(U,\mathcal {F}) =0$
(resp. $H^i(V,\mathcal {F})=0$) for all integers $i \ge 3$ and all 
coherent sheaves $\mathcal {F}$
on $U$ (resp. $V$). Since each fiber of the proper map
$\pi$ has dimension at most $1$, $R^j\pi _\ast (F)= 0$ for
all $j \ge 2$. Hence the Leray spectral sequence of $\pi$ gives an 
exact sequence
\begin{equation}\label{eqb4}
0 \to H^1(U,\pi _\ast (F)) \to H^1(V,F) \to H^0(U,R^1\pi _\ast (F)) 
\to H^2(U,\pi _\ast (F))
\end{equation}
(\cite{b}, p. 31). Hence if $H^1(V,F) = 0$ and $H^2(U,\pi _\ast 
(F))=0$, then $H^0(U,R^1\pi _\ast (F)) = 0$. Since $R^1\pi _\ast (F)$ 
is supported by $P$,
$H^0(U,R^1\pi _\ast (F)) = 0$ if and only if $R^1\pi _\ast (F) = 0$. 
The same relation is true if we blow up more than one point.
Hence we get the following observation. Let $A$ be a rank $r$ vector 
bundle on $X_s$. Assume $h^1(X_s,A(t;0,\dots ,0))$
for all $t \gg 0$. Notice that for fixed $A$ we have $H^1({\bf 
{P}}^2,f_{s \ast }(A)(t))=H^2({\bf {P}}^2,f_{s \ast } (A)(t))=0$ for 
$t \gg 0$, while
the integer $h^0({\bf {P}}^2,R^1f_{s _\ast }(A)(t))$ does not depend 
from $t$, because the sheaf $R^1f _{s \ast } (A)$
is supported by the finite set $\{P_1,\dots ,P_s\}$. We get $R^1f _{s 
\ast }(A)=0$.
Let
$b_{i,1} \ge \cdots \ge b_{i,r}$ be the splitting type of $A\vert 
D_i$. Since $R^1f_{s \ast }(A)=0$,
$b_{i,r} \ge 0$ for all $i$ (Remark \ref{a2}).
\end{remark}

\qquad {\emph {Proofs of Theorem \ref{a3} and of Remark 
\ref{a3.0}.}}.Obviously, (c) implies (a) (Remark
\ref{a1}). The projection formula gives that (c) implies (b).
Assume that (a) holds, but only assuming that $E$ is torsion free. The
line bundle $\mathcal {O}_{X_s}(2s;-1,\dots ,-1)$ is ample. Hence 
$\mathcal {O}_{X_s}(2s;0,\dots ,0)$
is the tensor product of an ample line bundle and of a line bundle 
with a non-zero section. Remark
\ref{c0} gives that $E$ is locally free. Remark \ref{b3} gives 
$R^1f_{s \ast }(E)=0$ and $b_{i,r} \ge 0$ for
all $i$. Fix a line $D\subset {\bf {P}}^2$ such that $\{P_1,\dots 
,P_s\}\cap D = \emptyset$.
Hence $D \cong f^{-1}(D)$ of $X_s$.
Let $t_1 \ge \cdots \ge t_r$ be the splitting type of $E\vert f^{-1}(D)$. Let
$\epsilon :H^0(X_s,E(-t_1;0,\dots ,0))\otimes \mathcal {O}_{X_s} \to 
E(-t_1;0,\dots ,0)$
denote the evaluation map.
Let $b$ the maximal integer such that
$1 \le b \le r$ and $t_b=t_1$. Hence $h^0(f^{-1}(D),E(-t_1)\vert 
f^{-1}(D)) = b$. Since $f^{-1}(D)\in \vert
\mathcal {O}_{X_s}(1;0,\dots ,0)\vert$, we have an exact sequence
\begin{equation}\label{eqa2}
0 \to E(-t-1;0,\dots ,0) \to E(-t;0,\dots ,0) \to E(-t;0,\dots 
,0)\vert f^{-1}(D) \to 0
\end{equation}
Look at the cases $t=t_1$ and $t=t_1+1$ of (\ref{eqa2}). Our 
cohomological assumption on $E$ gives $h^0(X_s,E(-t_1;0,\dots ,0) = b$
and that $\epsilon$ has rank $b$ at each point of $f^{-1}(D)$. Hence 
$\mbox{Im}(\epsilon )$ is a rank $b$
torsion free sheaf spanned by a $b$-dimensional linear space of 
global sections. Hence $\mbox{Im}(\epsilon )
\cong \mathcal {O}_{X_s}^{\oplus b}$. We also get that the restriction of the
inclusion $u: \mbox{Im}(\epsilon ) \to E(-t_1;0,\dots ,0)$ to the 
fiber over every $P\in X_s\backslash D_1\cup \cdots \cup
D_s$ has rank $b$. If $s=0$ we also get
that $\mathcal {O}_{{\bf {P}}^2}(t_1)^{\oplus b}$ is a subbundle $E'$ 
of $E$ such that $E/E'\vert D$ has splitting
type $t_{b+1}\ge \cdots \ge t_r$. After finitely many steps we get 
(unfortunately, with the classical proof)
the case $s=0$, i.e. that an ACM torsion free sheaf on ${\bf {P}}^2$ 
is a direct sum of line bundles. Hence we may assume
$s>0$. Since $R^1f_{s \ast }(E)=0$, the projection formula and the 
Leray spectral sequence
of $f_s$ give
$h^1({\bf {P}}^2,f_{s \ast }(E)(t)) =0$ for all $t\in \mathbb {Z}$. Hence
the torsion free sheaf $f_{s \ast }(E)$ is ACM. We saw in the proof 
of the case $s=0$
that $f_{s \ast }(E)$ is a direct sum of line bundles and in 
particular it is locally free. Thus $b_{i,j} =0$
for all $i\in \{1,\dots ,s\}$ and $j\in \{1,\dots ,r\}$
(Lemma \ref{b2}). Since $E$ is locally free, these equalities are
equivalent to the existence of a vector bundle $A$ on ${\bf {P}}^2$ 
such that $E \cong f_s^\ast (A)$.
Since $A \cong f_{s \ast }(f_s^\ast (A))$, we get that (a) implies 
(c) and (b). Assume (b). Lemma \ref{b2}
gives $b_{i,j} =0$
for all $i\in \{1,\dots ,s\}$ and $j\in \{1,\dots ,r\}$. Hence $E 
\cong f_s^\ast (f_{s \ast }(E))$. Thus
(b) implies (c).\\
Alternatively, one can use descent to get $E \cong f_s^\ast (M)$ for 
some torsion free sheaf $M$ on ${\bf {P}}^2$,
then show that $M$ is locally free using Remark \ref{c0} and then 
show that $M$ is isomorphic to a direct sum of line
bundles.\qed

\section{Rank $2$ vector bundles on ${\bf {P}}^2$}\label{S3}

\begin{proposition}\label{c1}
Let $E$ be a rank $2$ torsion free sheaf on ${\bf {P}}^2$ such that 
$h^1({\bf {P}}^2,E(t))=0$ for every even integer
$t$. Then either $E \cong \Omega _{{\bf {P}}^2}(z)$ for some odd integer $z$
or $E$ is a direct sum of two line bundles.
\end{proposition}

\begin{proof}
The Euler's sequence gives $h^1({\bf {P}}^2,\Omega _{{\bf {P}}^2}(t)) 
= 0$ for $t \ne 0$ and
$h^1({\bf {P}}^2,\Omega _{{\bf {P}}^2})=1$. Hence the `` if '' part 
is obvious. Now assume $h^1({\bf {P}}^2,E(t))=0$ for every even 
integer
$t$. Remark \ref{c0} gives that $E$ is locally free. Let $w$ be the 
first integer $t$ such that $h^0({\bf {P}}^2,E(t))\ne 0$. Fix any
$\sigma \in H^0({\bf {P}}^2,E(w))\backslash \{0\}$. The minimality of 
$w$ shows that $\sigma$ induces an exact sequence
\begin{equation}\label{eqc1}
0 \to \mathcal {O}_{{\bf {P}}^2}(t) \to E(w+t) \to \mathcal 
{I}_Z(c_1+2w+t) \to 0
\end{equation}
where $c_1:= c_1(E)$ and $Z$ is a zero-dimensional subscheme of ${\bf 
{P}}^2$. The minimaly of $w$ is equivalent to
$h^0({\bf {P}}^2,\mathcal {I}_Z(c_1+2w-1))=0$. Hence either $c_1+2w-1 
< 0$ or $z:= \mbox{length}(Z)
\ge (c_1+2w+1)(c_1+2w)/2$. First assume $w$ even. Taking $t=-2$ in 
(\ref{eqc1}) and use $h^2({\bf {P}}^2,\mathcal {O}_{{\bf {P}}^2}
(-2)) = 0$ we get
$h^1({\bf {P}}^2,\mathcal {I}_Z(c_1+2w-2))=0$. $E$ splits if and only if $z=0$.
Assume $z>0$. If $c_1+2w-1 < 0$, then $h^1({\bf {P}}^2,\mathcal 
{I}_Z(c_1+2w-2)) = z>0$. If $c_1+2w-1 \ge 0$
we get $h^1({\bf {P}}^2,\mathcal {I}_Z(c_1+2w-2)) \ge 
(c_1+2w+1)(c_1+2w)/2 - (c_1+2w)(c_1+2w-1)/2 = c_1+2w>0$,
contradiction. Now assume $w$ odd. Take $t=-3$ in (\ref{eqc1}). Since 
$h^2({\bf {P}}^2,\mathcal {O}_{{\bf {P}}^2}
(-3)) = 1$, we get that either $z=1$ or $c_1+2w =1$. If $c_1+2w=1$, 
then the minimality of $w$
and (\ref{eqc1}) gives $z=1$. Since $z \ge (c_1+2w+1)(c_1+2w)/2$, in 
both cases we get
$c_1+2w=1$. If $z=1$ and $c_1+2w=0$, then (\ref{eqc1}) gives that 
$E(w)$ is a stable rank two reflexive sheaf
with $c_1(E(w)) = 1$ and $c_2(E(w)) = 1$. It is well-known and easy 
to check that $\Omega _{{\bf {P}}^2}(2)$
is the only such vector bundle.
\end{proof}

Proposition \ref{c1} immediately implies the following result, which 
also follows from Beilinson spectral sequence.

\begin{corollary}\label{c2}
Fix an integer $m$.
Let $E$ be a rank $2$ torsion free sheaf on ${\bf {P}}^2$ such that 
$h^1({\bf {P}}^2,E(t))=0$
for all $t\in \mathbb {Z}\backslash \{m\}$. Then either $E \cong 
\Omega _{{\bf {P}}^2}(-m)$ or $E$ is a direct sum of two line bundles.
\end{corollary}

\begin{proposition}\label{c2.1}
Let $E$ be a rank $2$ torsion free sheaf on ${\bf {P}}^2$ such that 
$h^1({\bf {P}}^2,E(t))=0$ for
every integer
$t$ such that $t \equiv 0 \pmod{3}$. Let $w$ be the first integer $x$ 
such that $h^0({\bf {P}}^2,E(x))>0$.
Set $c_1:= c_1(E)$. Then $E$ is isomorphic to one of the following 
vector bundles:
\begin{itemize}
\item[(i)] a direct sum of two line bundles.
\item[(ii)] $\Omega _{{\bf {P}}^2}(2-w)$.
\item[(iii)] $c_1+2w = 2$, $E$ is stable and it fits in an exact sequence
\begin{equation}\label{eqc1.1}
0 \to \mathcal {O}_{{\bf {P}}^2} \to E(w) \to \mathcal {I}_Z(2) \to 0
\end{equation}
in which $Z$ is a curvilinear zero-dimensional scheme of length $3$ 
not contained in a line..
\end{itemize}
Conversely, any vector bundle $E$ as in (i), (ii) or (iii) has the 
property that $h^1({\bf {P}}^2,E(t))=0$ for
every integer
$t$ such that $t \equiv 0 \pmod{3}$. In case (iii) we have $h^1({\bf 
{P}}^2,E(z)) = 0$ for all $z \notin \{w-3,w-2\}$,
$h^1({\bf {P}}^2,E(w-2)) = 2$ and $ h^1({\bf {P}}^2,E(w-3)) =2$.
\end{proposition}

\begin{proof}
Remark \ref{c0} gives that $E$ is locally free. Let $b$ denote the 
only integer such that $1 \le b \le 3$
and $w \equiv -b \pmod{3}$. Fix any
$\sigma \in H^0({\bf {P}}^2,E(w))\backslash \{0\}$. The minimality of 
$w$ shows that $\sigma$ induces an exact sequence
\begin{equation}\label{eqc1.00}
0 \to \mathcal {O}_{{\bf {P}}^2}(t) \to E(w+t) \to \mathcal 
{I}_Z(c_1+2w+t) \to 0
\end{equation}
where $c_1:= c_1(E)$ and $Z$ is a locally complete intersection 
zero-dimensional subscheme of ${\bf {P}}^2$. The minimalityy of $w$ 
is equivalent to
$h^0({\bf {P}}^2,\mathcal {I}_Z(c_1+2w-1))=0$. Hence either $c_1+2w-1 
< 0$ or $z:= \mbox{length}(Z)
\ge (c_1+2w+1)(c_1+2w)/2$. If $z=0$, then $E$ is a direct sum of two 
line bundles.
Hence we may assume $z>0$. First assume $b=1$. Taking $t=-1$ in 
(\ref{eqc1.00}) and using $h^2({\bf {P}}^2,\mathcal {O}_{{\bf 
{P}}^2}(-1))=0$
we get $h^1({\bf {P}}^2,\mathcal {I}_Z(c_1+2w-1))=0$. Since $z>0$, we 
get $c_1+2w-1 \ge 0$
and $z \le (c_1+2w+1)(c_1+2w)/2$.
Hence $z = (c_1+2w+1)(c_1+2w)/2$. Take $t=-4$ in (\ref{eqc1.00}) and using
$h^2({\bf {P}}^2,\mathcal {O}_{{\bf {P}}^2}(-4))=3$ we get $h^1({\bf 
{P}}^2,\mathcal {I}_Z(c_1+2w-4)) \le 3$.
Since $c_1+2w-1 \ge 0$ and $z = (c_1+2w+1)(c_1+2w)/2$, we get $1 \le 
c_1+2w \le 2$. First assume $c_1+2w = 1$ and hence
$z=1$. Thus $c_2(E(w)) =1$ and $c_1(E(w)) = 1$. The exact sequence 
(\ref{eqc1.00}) gives the stability
of $E(w)$. Hence
we are in case (ii). If $c_1+2w = 2$, then we are in case (iii); here 
we use $h^0({\bf {P}}^2,\mathcal {I}_Z(c_1+2w-1))=0$
to see that $Z$ is contained in no line. Now assume $b=2$. Taking 
$t=-2$ in (\ref{eqc1.00}) and using
$h^2({\bf {P}}^2,\mathcal {O}_{{\bf {P}}^2}(-2))=0$
we get $h^1({\bf {P}}^2,\mathcal {I}_Z(c_1+2w-2))=0$. Hence $c_1+2w-2 
\ge 0$ and
$z \le (c_1+2w)(c_1+2w-1)/2$, contradicting the inequality $z
\ge (c_1+2w+1)(c_1+2w)/2$. Now assume $b=3$. Taking $t=-3$ in 
(\ref{eqc1.00}) and using
$h^2({\bf {P}}^2,\mathcal {O}_{{\bf {P}}^2}(-3))=1$
we get $h^1({\bf {P}}^2,\mathcal {I}_Z(c_1+2w-3)) \le 1$. Hence 
$c_1+2w-2 \ge 0$ and
$z \le (c_1+2w-1)(c_1+2w-2)/2 + 1$, contradicting the inequality $z
\ge (c_1+2w+1)(c_1+2w)/2$. The `` converse '' part is obvious for 
cases (i) and (ii). Take $E$ as in case (iii)
(i.e. as in (\ref{eqc1.00}) with $c_1+2w = 2$, $w \equiv -1 
\pmod{2}$, $\mbox{length}(Z)=3$ and $Z$ not contained
in a line), without
assuming the local freeness of $E$. Since $Z$ is not contained in a 
line, from (\ref{eqc1.00}) we get
$h^1({\bf {P}}^2,E(z)) = 0$ for all $z \ge w-1$. Now assume $E$ 
locally free. Serre duality gives
$h^1({\bf {P}}^2,E^\ast (y)) =0$ for all $y \le -2 -w$. Since 
$\mbox{rank}(E) = 2$ and $E$ locally free,
$E^\ast \cong E(-c_1)$. Since $c_1+2w = 2$, we get $h^1({\bf 
{P}}^2,E(z)) = 0$ for all $z \notin \{w-2,w-3\}$. From
(\ref{eqc1.00}) we get $h^1({\bf {P}}^2,E(w-2)) = 2$ and $2 \le 
h^1({\bf {P}}^2,E(w-3)) \le 3$. By (\ref{eqc1.00})
we have $ h^1({\bf {P}}^2,E(w-3)) =2$ if and only if $h^2({\bf 
{P}}^2,E(w-3)) =0$. Since $c_1(E(w-3)) = -4$, $(E(w-3))^\ast
\cong E(w-3)(4)$. Thus $h^2({\bf {P}}^2,E(w-3)) = h^0({\bf 
{P}}^2,E(w-2)) =0$. Hence $ h^1({\bf {P}}^2,E(w-3)) =2$.
Now assume that the length $3$ scheme $Z$ is curvilinenar, i.e. it is 
not the first infinitesimal neighborhood of a point
of ${\bf {P}}^2$. Since $h^0({\bf {P}}^2,\mathcal {O}_{{\bf 
{P}}^2}(-1)) =0$, the Cayley-Bacharach
condition is trivially satisfied and hence a general extension 
(\ref{eqc1.00}) with $c_1+2w =0$ and this curvilinear scheme $Z$
has locally free middle term.
\end{proof}

\begin{remark}\label{c2.1.0}
Let $M({\bf {P}}^2,c_1,c_2)$ denote the moduli space of rank $2$ 
wector bundle on ${\bf {P}}^2$ with
Chern classes $c_1,c_2$. $M({\bf {P}}^2,0,2)$ is non-empty, 
irreducible and $5$-dimensional. Take $E$
as in case (iii) of Proposition \ref{c2.1}. We have $c_1(E) =0$ and 
$c_2(E) = c_2(E(1))-c_1(E(1)) + 1 =2$
(\cite{h2}). We saw that $E$ is stable, i.e. $E\in M({\bf 
{P}}^2,0,2)$. Take any $F\in M({\bf {P}}^2,0,2)$.
Since $c_1(F(1)) =2$ and $c_2(F(1)) =3$, Riemann-Roch gives
$\chi (F(1)) = (2+3)2/2 +2 -3 = 4>0$. The stability of $F$ gives 
$h^0({\bf {P}}^2,F)=0$.
Hence $F$ fits in the extension (\ref{eqc1.1}), i.e. $F$ is described 
by case (iii) of Proposition \ref{c2.1}.
\end{remark}

\begin{remark}\label{c2.2}
Fix an integer $a \ge 4$ and a rank $2$ torsion free sheaf $E$ on 
${\bf {P}}^2$ such that $h^1({\bf {P}}^2,E(t)) =0$
for all integers $t$ such that $t \equiv 0 \pmod{a}$. Here we will 
see that this assumption is very restrictive,
but that it seems hopeless to try to classify all such sheaves $E$.
Remark \ref{c0} gives that $E$ is locally free. Set $c_1:= c_1(E)$. Let
$w$ be the first integer such that $h^0({\bf {P}}^2,E(w)) >0$. Hence 
we have an exact sequence
(\ref{eqc1.00}) with $Z$ a zero-dimensional locally complete 
intersection subscheme. The minimality
of $w$ gives $h^0({\bf {P}}^2,\mathcal {I}_Z(c_1+2w-1))=0$. Hence 
either $c_1+2w-1 <0$
or $z:= \mbox{length}(Z) \ge (c_1+2w+1)(c_1+2w)/2$. If $z=0$, then 
$E$ splits. Hence we assume $z>0$.
First assume $c_1+2w < 0$. Look at (\ref{eqc1.1}). We get $h^1({\bf 
{P}}^2,E(t))>0$ for all $t < w$
such that either $t+w \le -2$ or $z > (-t-w-1)(-t-w-2)/2$. Given any 
locally complete intersection $Z$ in the case
$c_1+2w < 0$
we get an extension (\ref{eqc1.00}) with locally free middle term.
Now assume $c_1+2w \ge 0$. Let
$b$ denote the only integer such that $w \equiv -b \pmod{a}$ and $1 
\le b \le a$. First assume $b=1$ as in the cases
$a=2$ and $a=3$ we get $z = (c_1+2w+1)(c_1+2w)/2$ and $h^i({\bf 
{P}}^2,\mathcal {I}_Z(c_1+2w-1))=0$, $i=1,2$.
Taking $t = -1-a$ and using $h^2({\bf {P}}^2,\mathcal {O}_{{\bf 
{P}}^2}(-1-a)) = a(a-1)/2$ we get
$z \le \epsilon + a(a-1)/2$, where $\epsilon = 
(c_1+2w-a+1)(c_1+2w-a)/2$ if $c_1+2w \ge a+1$ and
$\epsilon = 0$ if $c_1+2w \le a$. We first get $c_1+2w \le a$ and 
then we get $c_1+2w \le a-1$. Now assume
$2 \le b \le a$. Taking $t=w-b$ we get $z \le \eta + (b-1)(b-2)/2$, 
where $\eta = 0$ if $c_1+2w < b$
and $\eta = (c_1+2w-b+2)(c_1+2w-b+1)/2$ if $c_1+2w \ge b$.
\end{remark}

We raise the following question.

\begin{question}\label{cc00}
Fix an integer $a \ge 4$ and a rank $2$ torsion free sheaf $E$ on 
${\bf {P}}^2$ such that $h^1({\bf {P}}^2,E(t)) =0$
for all integers $t$ such that $t \equiv 0 \pmod{a}$. Is it true that 
$h^1({\bf {P}}^2,E(t)) \ne 0$ for at most
$a-1$ consecutive integers?
\end{question}

\providecommand{\bysame}{\leavevmode\hbox to3em{\hrulefill}\thinspace}

\end{document}